\documentclass[12pt]{amsart}

\usepackage{amsmath,amssymb}
\usepackage[dvipdfmx]{graphicx}
\usepackage{graphicx, tabularx}

\begin{document}

\setlength{\topmargin}{1cm} 

\newfont{\ef}{eufm10 scaled\magstep1}
\newfont{\efs}{eufm8}
\newfont{\bbl}{msbm10 scaled\magstep1}
\newfont{\cl}{cmsy10 scaled\magstep1}
\newfont{\cls}{cmsy8}

%
\newcommand{\St}{\mbox{\ef T}}
\newcommand{\XG}{\mbox{\ef X}}
\newcommand{\CG}{\mbox{\ef C}}
\newcommand{\DG}{\mbox{\ef D}}
\newcommand{\UG}{\mbox{\ef U}}
\newcommand{\UGs}{\mbox{\efs U}}
\newcommand{\LG}{\mbox{\ef L}}
\newcommand{\MG}{\mbox{\ef M}}
\newcommand{\NGG}{\mbox{\ef N}}
\newcommand{\PG}{\mbox{\ef P}}
\newcommand{\WG}{\mbox{\ef W}}
\newcommand{\XGG}{\mbox{\ef X}}
\newcommand{\XGGs}{\mbox{\efs X}}
\newcommand{\Zh}{\mbox{\bbl Z}}
\newcommand{\Nat}{\mbox{\bbl N}}
\newcommand{\Rd}{\mbox{\cl R}}
\newcommand{\Gh}{\mbox{\cl G}}
\newcommand{\Ghs}{\mbox{\cls G}}
\newcommand{\HGh}{\mbox{\cl H}}
\newcommand{\HGhs}{\mbox{\cls H}}
\newcommand{\Sy}{\mbox{\cl S}}
\newcommand{\Sys}{\mbox{\cls S}}
\newcommand{\Ty}{\mbox{\cl T}}
\newcommand{\Tys}{\mbox{\cls T}}
\newcommand{\Nc}{\mbox{\cl N}}
\newcommand{\Wsy}{\mbox{\cl W}}
\newcommand{\Vt}{\mbox{\cl V}}
\newcommand{\Vts}{\mbox{\cls V}}
\newcommand{\Ed}{\mbox{\cl E}}
\newcommand{\Bs}{\mbox{\cl B}}
\newcommand{\Ms}{\mbox{\cl M}}
\newcommand{\Ds}{\mbox{\cl D}}

\newtheorem{theorem}{Theorem}[section]
\newtheorem{lemma}[theorem]{Lemma}
\newtheorem{proposition}[theorem]{Proposition}
\newtheorem{corollary}[theorem]{{\bf Corollary}}

\theoremstyle{definition}
\newtheorem{definition}[theorem]{Definition}
\newtheorem{example}[theorem]{Example}
\newtheorem{xca}[theorem]{Exercise}

\theoremstyle{remark}
\newtheorem{remark}[theorem]{Remark}






\begin{center}

{\Large {\bf Uncountable locally free groups\\
and their group rings}}
\bigskip

{\small TSUNEKAZU NISHINAKA}
\medskip



{\Small University of Hyogo\\
8-2-1 Gakuen Nishimachi Nishiku \\
Kobe-City 651-2197 Japan\\
Email: nishinaka@econ.u-hyogo.ac.jp}\\


\end{center}




\renewcommand{\thefootnote}{\fnsymbol{footnote}}

\footnote[0]{
This research was partially supported by Grants-in-Aid for 
Scientific Research (KAKEN) under grant no. 26400055\\
}

\begin{center}
\begin{tabularx}{13cm}{X}
{\Small {\sc Abstract.}
In this note, we show that an uncountable locally free group,
and therefore every locally free group, has a free subgroup 
whose cardinality is the same as that of $G$.
This result directly improve the main result in \cite{Ni11}
and establish the primitivity of group rings
of locally free groups.
}
\end{tabularx}
\end{center}



\section{INTRODUCTION}

A group $G$ is called locally free 
if all of its finitely generated subgroups are free.
As a consequence of Nielsen-Schreier theorem,
a free group is always locally free.
If the cardinality $|G|$ of $G$
is countable,
then $G$ is locally free
if and only if $G$ is an ascending union of free groups.
In particular,
$G$ is a locally free group which is not free
provided that it is 
a properly ascending union of non-abelian free groups of bounded finite rank.
In fact, in this case,
$G$ is infinitely generated and Hopfian and so it is not free
(also see \cite{Taka} and \cite{Hig}).
If $|G|$ is uncountable,
that is $|G|>\aleph_0$,
then it was studied in the context of almost free groups,
and it is also known that 
there exists an uncountable locally free group which is not free
(\cite{Hig51}).

Now, clearly,
if $G$ is a locally free group
with $|G|=\aleph_0$,
then $G$ has a free subgroup 
whose cardinality is the same as that of $G$.
In the present note,
we shall show that it is true for locally free groups of
any cardinality.
In fact, we shall prove the following theorem:

\begin{theorem}\label{MainTH}
If $G$ is a locally free group with 
$|G|>\aleph_0$,
then for each finitely generated subgroup $A$ of $G$,
there exists a subgroup $H$ of $G$ with $|H|=|G|$
such that $AH\simeq A*H$,
the free product of $A$ and $H$.

In particular,
$G$ has a free subgroup of the same cardinality as that of $G$.
\end{theorem}

Theorem~\ref{MainTH}
means that there is no need to
assume the existence of free subgroups in \cite[Theorem 1]{Ni11}.
That is, we can improve the theorem
and establish the primitivity of group rings
of locally free groups,
where a ring $R$ is (right) primitive 
provided it has a faithful irreducible (right) $R$-module.

\section{Proof of the theorem}

In order to prove Theorem~\ref{MainTH},
we prepare necessary notations 
and some lemmas which include a result due to Mal'cev \cite{Mal}.
Some of them might be trivial for experts
but we include their proofs for completeness.

For a finitely generated subgroup $H$ of a locally free group $G$,
$\mu_{G}(H)$ is defined to be the least positive integer $m$
such that $H\subseteq F_m$ for some
free subgroup $F_m$ of rank $m>0$ in $G$.
The rank $r(G)$ of $G$ is defined to be the maximum element
in 
$$\{\mu_{G}(H)\ |\ H 
\mbox{ is a finitely generated subgroup in } G\}$$
or $r(G)=\infty$.
We should note that
$r(G)$ is finite if and only if
for each finitely generated subgroup $H$ of $G$,
there exists a free subgroup $N$ of rank $r(G)$
such that $H\subseteq N$.
We should also note that for subgroups $H$, $M$ and $N$ of $G$
with $H\subseteq M\subseteq N$,
$\mu_{M}(H)\geq \mu_{N}(H)$ holds.

\begin{lemma}\label{M_lemma}
Let $G$ be a locally free group,
$D$ a finitely generated subgroup of $G$
with $\mu_{G}(D)=m$
and $M$ a subgroup of $G$ such that $D\subseteq M$ and $r(M)=m$.
For $g\in G\setminus M$,
set $M_g=M\langle g\rangle$;
the subgroup of $G$ generated by $g$ and the elements in $M$,
and let $r(M_g)=n$.

Then,
if $n>m$,
then $n=m+1$ 
and there exists a free subgroup $F_m$ of rank $m$
in $M$ such that
$D\subseteq F_m$ and $F_m\langle g\rangle$
is isomorphic to the free product $F_m*\langle g\rangle$.
\end{lemma}

\begin{proof}
Since $r(M_g)=n$,
there exists a finitely generated subgroup $C$ of $M_g$
such that $\mu_{M_g}(C)=n$.
Since $C$ is finitely generated,
it can be easily seen that there exists
finite number of elements $a_1,\ldots,a_l$ in $M$
such that $C\subseteq\langle a_1,\ldots,a_l, g\rangle$.
We have then that
there exists a free subgroup $F_m$ of rank $m$ in $M$
such that $\langle a_1,\ldots,a_l\rangle D\subseteq F_m$
because of $r(M)=m$.
Since $C\subseteq F_m\langle g\rangle$ and $\mu_{M_g}(C)=n$,
we see that $r(F_m\langle g\rangle)\geq n$.
On the other hand,
$r(F_m\langle g\rangle)\leq m+1$ because of $g\not\in F_m$.
Combining these with the assumption $n>m$,
we get that $n=m+1=r(F_m\langle g\rangle)$,
which implies $F_m\langle g\rangle\simeq F_m*\langle g\rangle$.
\end{proof}

For a locally free group of finite rank,
the following result due to Mal'cev is well-known.

\begin{lemma}\mbox{{\rm(See \cite{Mal})}}\label{Mal_L}
If $G$ is a locally free group of finite rank,
then the cardinality of $G$ is countable; namely $|G|=\aleph_0$.
\end{lemma}

On the other hand,
if $G$ is not of finite rank,
then we have the following property:

\begin{lemma}\label{Subgr}
If $G$ is a locally free group whose rank is not finite,
then for each finitely generated subgroup $A$ of $G$,
there exists an element $x\in G$ with $x\not\in A$,
such that $A\langle x\rangle\simeq A*\langle x\rangle$,
the free product of $A$ and $\langle x\rangle$.
\end{lemma}

\begin{proof}
Let $A$ be a finitely generated subgroup of $G$.
We have then that $A$ is a free group of finite rank,
because $G$ is locally free.
Since the rank of $G$ is not finite,
there exists a free subgroup $F$ of $G$ 
such that $A\subsetneq F$ and $r(F)>r(A)$.
Let $A=\langle y_1,\ldots, y_l\rangle$
and $F=\langle x_1,\ldots, x_m\rangle$,
where $l=r(A)$ and $m=r(F)$.
If for $A_1=\langle y_1,\ldots, y_l, x_1\rangle$,
$r(A_1)=l+1$,
then $A_1\simeq A*\langle x_1\rangle$.
If $r(A_1)\leq l$ then there exists $i\in\{2,\ldots,m\}$
such that $r(A_{i})=r(A_{i-1})+1$,
where $A_i=\langle y_1,\ldots, y_l, x_1,\ldots, x_i\rangle$.
We have then that $A_i\simeq A_{i-1}*\langle x_i\rangle$.
Since $A\subseteq A_{i-1}$,
we have thus seen that  $A\langle x\rangle\simeq A*\langle x\rangle$
for some $x\in G\setminus A$.
\end{proof}

Let $G=A*B$ be the free product of $A\ne 1$ and $B\ne 1$.
Clearly, if $|G|=\aleph_0$,
then $G$ has a free subgroup whose cardinality is the same as that of $G$.
If $|G|>\aleph_0$, then either $|A|=|G|$ or $|B|=|G|$,
say $|A|=|G|$.
Let $I$ be a set with $|I|=|A|$,
and for each $i\in I$,
let $a_i$ be in $A$ such that $a_i\ne a_j$ for $i\ne j$.
We have then that for $1\ne b\in B$,
the elements $(a_ib)^2$ over $i\in I$ freely generate the subgroup of $G$
whose cardinality is the same as that of $G$.
Hence we have

\begin{lemma}\label{FSofFPr}
If $G=A*B$ is the free product of $A$ and $B$,
then $G$ has a free subgroup
whose cardinality is the same as that of $G$.
\end{lemma}

We are now read to prove Theorem~\ref{MainTH}.

\begin{proof}[Proof of Theorem~\ref{MainTH}]
Let $A$ be a finitely generated subgroup of $G$.
We set
$$\Bs=\{ B\ |\ B \mbox{ is a non-trivial subgroup of } G 
\mbox{ such that } AB\simeq A*B\}.$$
Since $G$ is locally free,
$A$ is a free group of finite rank.
By assumption,
$|G|>\aleph_0$,
and so the rank of $G$ is not finite by Lemma \ref{Mal_L}.
Hence, by Lemma~\ref{Subgr},
there exists an element $g\in G\setminus A$ such that
$A\langle g\rangle\simeq A*\langle g\rangle$,
whence $\langle g\rangle\in \Bs$; thus $\Bs\ne \emptyset$.
Let $B_1\subseteq B_2\subseteq\cdots\subseteq B_i\subseteq\cdots$
be a chain of $B_i$'s in $\Bs$,
and let $B^{*}=\bigcup_{i=1}^{\infty}B_i$.
We can see that $B^{*}$ belongs to $\Bs$.
In fact, if not so,
then $AB^{*}\not\simeq A*B^{*}$, and so
there exists a finitely generated subgroup $C$ of $B^{*}$
such that $AC\not\simeq A*C$.
However,
because $C$ is finitely generated in $B^{*}$,
we have $C\subseteq B_i$ for some $i$,
which implies $AB_i\not\simeq A*B_i$, a contradiction.
We have thus shown that $(\Bs, \subseteq)$
is an inductively ordered set.
By Zorn's lemma, 
there exists a maximal element $H$ in $(\Bs, \subseteq)$.
We shall show $|H|=|G|$,
which completes the proof of the theorem.
In fact, $AH\simeq A*H$ 
and it has also a free subgroup
whose cardinality is the same as that of $G$
by Lemma~\ref{FSofFPr}.

Suppose, to the contrary, that $|H|<|G|$.
Set $N=AH (\simeq A*H)$, and for a finitely generated subgroup $D$ of $N$
with $\mu_{G}(D)=m$,
let $\Ms(D)$ be the set of subgroups $M$ of $G$
such that $D\subseteq M$ and $r(M)=m$.
We can see that $(\Ms(D),\subseteq)$ is an inductively ordered set
as follows:
Since $m=\mu_{G}(D)$, there exists a free subgroup $F_m$ of rank $m$ in $G$
such that $D\subseteq F_m$.
Hence $F_m\in\Ms(D)$, whence $\Ms(D)\ne\emptyset$.
Let $M_1\subseteq M_2\subseteq\cdots\subseteq M_i\subseteq\cdots$
be a chain of $M_i$'s in $\Ms$,
and let $M^{*}=\bigcup_{i=1}^{\infty}M_i$.
Clearly, $D\subseteq M^{*}$.
By the definition of the rank,
$r(M^{*})\geq \mu_{M^{*}}(D)$.
Since $M^{*}\subseteq G$,
we have that $\mu_{M^{*}}(D)\geq\mu_{G}(D)=m$;
thus $r(M^{*})\geq m$.
On the other hand, by the definition of $r(M^{*})$,
there exists a finitely generated subgroup $C$ of $M^{*}$
such that $\mu_{M^{*}}(C)=r(M^{*})$.
Since $C$ is finitely generated,
there exists $i>0$ such that
$C\subseteq M_i$, and then $\mu_{M_i}(C)\leq r(M_i)$,
which implies that
$r(M^{*})=\mu_{M^{*}}(C)\leq\mu_{M_i}(C)\leq r(M_i)=m$; 
thus $r(M^{*})\leq m$.
Hence we have $r(M^{*})=m$. We have thus proved that $M^{*}\in \Ms(D)$
and that $(\Ms(D),\subseteq)$ 
is an inductively ordered set.

Again by Zorn's lemma, 
there exists a maximal element $M(D)$ in $(\Ms(D), \subseteq)$.
Let $L=\bigcup_{D\in\Ds}M(D)$,
where $\Ds$ is the set consisting of 
all finitely generated subgroups of $N$.
Since $r(M(D))$ is finite for each $D\in\Ds$,
it follows from Lemma \ref{Mal_L}
that $|M(D)|=\aleph_0$ for each $D\in\Ds$.
Hence we have $|L|<|G|$ because
$|\Ds|=|N|<|G|$.
In particular,
there exists $g\in G$ such that $g\not\in L$.
Note that $g\not\in N$ because of $N\subset L$,
and so $g\not\in H$.
We shall show that $N\langle g\rangle\simeq N*\langle g\rangle$.
In order to do this,
it suffices to show that for each $D\in\Ds$, 
$D\langle g\rangle\simeq D*\langle g\rangle$ holds.

Let $r(M(D))=m$ for $D\in\Ds$
and let $M_g=M(D)\langle g\rangle$.
Since $D\subseteq M_g$ and $\mu_{M_g}(D)\geq\mu_G(D)=m$,
it follows that $r(M_g)\geq m$.
Moreover,
since $g\not\in M(D)$,
we have $M(D)\subsetneq M_g$.
Hence
the maximality of $M(D)$
implies $r(M_g)>m$.
It follows from Lemma~\ref{M_lemma}
that there exists a free subgroup $F_m$ of rank $m$ in $M(D)$ 
such that $D\subseteq F_m$
and $F_m\langle g\rangle\simeq F_m*\langle g\rangle$.
Hence we have $D\langle g\rangle\simeq D*\langle g\rangle$.

We have thus shown that $N\langle g\rangle\simeq N*\langle g\rangle$,
which contradicts the maximality of $H$
because $N*\langle g\rangle=A*(H*\langle g\rangle)$. 
This completes the proof of the theorem.
\end{proof}

Theorem~\ref{MainTH}
shows that the assumption on 
existence of free subgroups in \cite[Theorem 1]{Ni11}
can be dropped.
That is, we have the following theorem:

\newpage

\begin{theorem}\label{LFG}
Let $G$ be a non-abelian locally free group.
If $R$ is a domain with $|R|\leq |G|$,
then $RG$ is primitive.

In particular, $KG$ is primitive for any field $K$.
\end{theorem}


\end{document}